\documentclass{article}

\title{Torsion groups of elliptic curves over $\mathbb{Q}(\mu_{p^{\infty}})$}
\author{Tomislav Gu\v{z}vi\'{c}}
\address{Department od Mathematics, University of Zagreb, Bijeni\v{c}ka cesta 30, Zagreb, Croatia}
\email{tguzvic@math.hr}
\urladdr{https://web.math.pmf.unizg.hr/~tguzvic/}
\author{Borna Vukorepa}
\address{Department of Mathematics, University of Zagreb, Bijeni\v{c}ka cesta 30, Zagreb, Croatia}
\email{bovukor@math.hr}
\urladdr{https://www.math.pmf.unizg.hr/hr/borna-vukorepa}

\begin{document}
\begin{abstract}
Let $E/\Q$ be an elliptic curve and $p \in \{5,7,11 \}$ be a prime. We determine the possibilities for $E(\mathbb{Q}(\zeta_{p}))_{tors}$. Additionally, we determine all the possibilities for $E(\mathbb{Q}(\zeta_{16}))_{tors}$ and $E(\mathbb{Q}(\zeta_{27}))_{tors}$. Using these results we are able to determine the possibilities for $E(\mathbb{Q}(\mu_{p^{\infty}}))_{tors}$.
\end{abstract}
\maketitle
\section{Introduction}
 Let $K$ be a number field such that $[K:\mathbb{Q}]=d$ and let $E/K$ be an elliptic curve. A celebrated theorem of Mordell and Weil shows that $E(K)$ is a finitely generated abelian group. Therefore this group can be decomposed as $E(K)=E(K)_{tors} \oplus \mathbb{Z}^{r}$, $r \ge 0$. It is known that $E(K)_{tors}$ is of the form $C_{m} \oplus C_{n}$ for two positive integers $m,n$ such that $m$ divides $n$, where $C_m$ and $C_n$ denote cyclic groups of order $m$ and $n$, respectively.
\par
One of the goals in the theory of elliptic curves is the classification of torsion groups of elliptic curves defined over various fields.
\par
Let $d$ be a positive integer. Define $\Phi(d)$ to be the set of possible isomorphism classes of groups $E(K)_{tors}$, where $K$ runs through all number fields $K$ of degree $d$ and $E$ runs through all elliptic curves over $K$. In \cite{21}, Merel proved that $\Phi(d)$ is finite for all positive integers $d$. The set $\Phi(1)$ can be seen in Theorem \ref{Theorem 2.1} and was determined by Mazur \cite{mazurtorzija}.
\begin{tm}[Mazur, \cite{mazurtorzija}]
\label{Theorem 2.1}
Let $E/\mathbb{Q}$ be an elliptic curve. Then

\[ E(\mathbb{Q})_{tors} \cong
\begin{cases} 
      C_m, & m=1,...,10, 12, \\
      C_{2} \oplus C_{2m}, & m=1,...,4.
   \end{cases} \]
\end{tm}
The set $\Phi(2)$ has been determined by Kenku, Momose and Kamienny \cite{kenkumomose}, \cite{kamienny}. Derickx, Etropolski, Hoeij, Morrow and Zureick-Brown have determined $\Phi(3)$ in \cite{hoeij}.
\par  
 Define $\Phi^{\mathrm{CM}}(d)$ to be the set of possible isomorphism classes of groups $E(K)_{tors}$, where $K$ runs through all number fields $K$ of degree $d$ and $E$ runs through all elliptic curves with complex multiplication ($\mathrm{CM}$). The set $\Phi^{\mathrm{CM}}(1)$ has been determined by Olson in \cite{32} and $\Phi^{\mathrm{CM}}(d)$ for $d=2,3$ by Zimmer and his collaborators in \cite{80}, \cite{81} and \cite{82}.
    The sets $\Phi^{\mathrm{CM}}(d)$, for $4\le d \le 13$ have been determined by Clark, Corn, Rice and Stankewicz in \cite{30}. Bourdon, Pollack and Stankewicz have determined torsion groups of $\mathrm{CM}$ elliptic curves over odd degree number fields in \cite{bourdonpollack}.
\par
 Define $\Phi_{\mathbb{Q}}(d) \subseteq \Phi(d)$ to be the set of possible isomorphism classes of groups $E(K)_{tors}$, where $K$ runs through all number fields $K$ of degree $d$ and $E$ runs through all elliptic curves defined over $\mathbb{Q}$. For $d=2,3$, the sets $\Phi_{\mathbb{Q}}(d)$ have been determined by Najman \cite{najman_ratcubic}.
 
 \begin{tm} \label{najmanquadratic}
Let $E/\mathbb{Q}$ be an elliptic curve and $K/\mathbb{Q}$ a quadratic extension. Then $E(K)_{tors}$ is isomorphic to the one of the following groups:
    \begin{gather*}
    C_{m}, \quad m = 1, 2, \dots, 9, 10, 12, 15, 16\\
    C_{2} \oplus C_{2m}, \quad m = 1, 2, 3, 4, 5, 6\\
    C_{3} \oplus C_{3m}, \quad m = 1, 2\\
    C_{4} \oplus C_{4}.
    \end{gather*}
     $C_{15}$ is the only group which appears in only finitely many cases, and only over the extensions $\mathbb{Q}(\sqrt{5})$ and $\mathbb{Q}(\sqrt{-15})$.
\end{tm}

\begin{tm} \label{najmancubic}
Let $E/\mathbb{Q}$ be an elliptic curve and $K/\mathbb{Q}$ a cubic extension. Then $E(K)_{tors}$ is isomorphic to the one of the following groups:
    \begin{gather*}
    C_{m}, \quad m = 1, 2, \dots, 10, 12, 13, 14, 18, 21\\
    C_{2} \oplus C_{2m}, \quad m = 1, 2, 3, 4, 7.
    \end{gather*}
     $C_{21}$ is the only group which appears in only finitely many cases, and only over the extension $\mathbb{Q}(\zeta_{9})^{+}$.
\end{tm}

The set $\Phi_{\mathbb{Q}}(4)$ has been determined by Chou \cite{chou2} and Gonz\'alez-Jim\'enez and Najman \cite{growth}. The set $\Phi_{\mathbb{Q}}(5)$ has been determined by Gonz\'alez-Jim\'enez in \cite{38}. Gonz\'alez-Jim\'enez and Najman have also proved that $\Phi_{\mathbb{Q}}(p)=\Phi(1)$ for primes $p \ge 7$ in \cite{growth}. For an odd prime $\ell$ and a positive integer $d$, Propp \cite{90} has determined when there exists a degree $d$ number field $K$ and an elliptic curve $E/K$ with $j(E) \in \mathbb{Q} \setminus \{ 0, 1728 \}$ such that $E(K)_{tors}$ contains a point of order $\ell$.

\smallskip

Let $\mu_n$, for positive integer $n$, be the set of all complex numbers $\omega$ such that $\omega^n = 1$. Note that for a prime number $p$ we have that $\Q(\mu_p) = \Q(\zeta_p)$, where $\zeta_p$ is, as usual, $p$\textsuperscript{th} primitive root of unity.

For a prime number $p$, we define a set $\mu_{p^\infty}$ as the set of all complex numbers $\omega$ for which there exists non-negative integer $k$ such that $\omega^{p^k} = 1$. Note that $\Qz$ is the set $\Q$ extended with all $p^{n^\text{th}}$primitive roots of unity.

In \cite{guzkri}, the authors considered the following problem: assume that $E/\mathbb{Q}$ is an elliptic curve, $p$ a prime number and $K=\Q(\mu_p^{\infty})$. They show that the torsion subgroup of $E$ grows only over small subfields of $K$. More precisely, they showed the following:

\begin{tm} \label{teo:rast_qzetap}
Let $E/\Q$ be an elliptic curve, then for a prime number $p \geq 5$ it holds that
\[E(\Qz)_\tors = E(\Qz[p])_\tors.\]
Furthermore,
\[E(\Qz[3][\infty])_\tors = E(\Qz[3^3])_\tors \qquad \text{and} \qquad E(\Qz[2][\infty])_\tors = E(\Qz[2^4])_\tors.\]
\end{tm}
\begin{remark}
This result is ``the best possible''. For $E = \href{http://www.lmfdb.org/EllipticCurve/Q/27a4}{27a4}$ we have that
\[E(\Qz[3^2])_\tors = C_{9} \subsetneq C_{27} = E(\Qz[3^3])_\tors\]
and for $E = \href{http://www.lmfdb.org/EllipticCurve/Q/32a4}{32a4}$ it holds that
\[E(\Qz[2^3])_\tors = C_{2} \oplus C_{4} \subsetneq C_{2} \oplus C_{8} = E(\Qz[2^4])_\tors.\]
\end{remark}

It becomes natural to ask how can the torsion group of $E/\mathbb{Q}$ grow when we consider the base change $E/\mathbb{Q}(\zeta_p)$. This becomes much harder then it seems as $p$ grows because our methods sometimes rely on pure computation.
\\
Our results are the following theorems.

\begin{tm}
Let $E/\Q$ be an elliptic curve. Then $E(\Q(\zeta_{16}))_\tors$ is one of the following groups (apart from those in Mazur's theorem):
\[C_{4} \oplus C_{4} \quad \href{https://www.lmfdb.org/EllipticCurve/Q/15a1}{(15a1)}, \qquad C_{2} \oplus C_{10} \quad \href{https://www.lmfdb.org/EllipticCurve/Q/2112bd2}{(2112bd2)}.\]

\end{tm}

\begin{tm} \label{zeta27}
Let $E/\Q$ be an elliptic curve. Then $E(\Q(\zeta_{27}))_\tors$ can be only one of the next two groups (apart from those in Mazur's theorem):
\[C_{3} \oplus C_{3}, \qquad C_{3} \oplus C_{6}, \qquad C_{3} \oplus C_{9}, \qquad C_{21}, \qquad C_{27}.\]
\end{tm}

Additionally, in Lemma \ref{petsedam} we give a description of the set of possible group structures $E(\mathbb{Q}(\zeta_p))_{tors}$ can be isomorphic to for $p=5,7$ and $11$, where $E/\mathbb{Q}$ is an elliptic curve.
\\
We discuss further attempts to classify $E(\mathbb{Q}(\zeta_p))_{tors}$, for arbitrary prime number $p$.
\\
Magma \cite{magma} code used in this paper can be found \href{https://github.com/brutalni-vux/TorsionCyclotomic?fbclid=IwAR32kNBmGbzdULcwWpmXhP5QZ0Ap4QjZf4vqSdbFJ_9dWko6d16W3YtRQpQ}{here.}

\section{Notation and auxiliary results}
Let $E/F$ be an elliptic curve defined over a number field $F$.
There exists an $F$-rational cyclic isogeny $\phi \colon E \to E'$ of degree $n$ if and only if $\langle P \rangle$, where $P \in E(\overline{F})$ is a point of order $n$, is a $\Gal(\overline{F}/F)$-invariant group; in this case we say that $E$ has an $F$-rational $n$-isogeny.
When $F=\mathbb{Q}$, the possible degrees of $n$-isogenies of elliptic curves over $\mathbb{Q}$ are known by the following theorem.

\begin{tm}[Mazur \cite{mazurizog}, Kenku \cite{kenku1}, \cite{kenku2}, \cite{kenku3}, \cite{kenku4}]
\label{Theorem 2.3}
Let $E/\mathbb{Q}$ be an elliptic curve with a rational $n$-isogeny. Then \[ n \in \{ 1,...,19,21, 25,27, 37, 43, 67, 163 \} .\]
There are infinitely many elliptic curves (up to $\overline{\mathbb{Q}}$-isomorphism) with a rational $n$-isogeny over $\mathbb{Q}$ for \[ n \in \{
{1, . . . , 10, 12, 13, 16, 18, 25} \} \] and only finitely many for all the other $n$. If $E$ does not have complex multiplication, then $n \le 18$ with $n \neq 14$ or $n \in \{21, 25, 37 \}.$
\end{tm}

\subsection*{Galois representations}
Let $E/\mathbb{Q}$ be an elliptic curve and let $n$ a positive integer. The field $\mathbb{Q}(E[n])$ is the number field obtained by adjoining to $\mathbb{Q}$ all the $x$ and $y$-coordinates of the points of $E[n]$. The absolute Galois group $\Gal(\overline{\mathbb{Q}}/\mathbb{Q})$ acts on $E[n]$ by its action on the coordinates of the
points, inducing a mod $n$ Galois representation attached to $E$:
\[ \rho_{E,n}: \Gal(\overline{\mathbb{Q}}/\mathbb{Q}) \to \Aut(E[n]) .\]
After we fix a basis for the $n$-torsion, we can identify $\Aut(E[n])$ with $\GL_{2}(\mathbb{Z}/n\mathbb{Z})$.
This means that we can consider $\rho_{E,n}(\Gal(\overline{\mathbb{Q}}/\mathbb{Q}))$ as a subgroup of $\GL_{2}(\mathbb{Z}/n\mathbb{Z})$, uniquely determined up to conjugacy. We shall denote $\rho_{E,n}(\Gal(\overline{\mathbb{Q}}/\mathbb{Q}))$ by $G_{E}(n)$. Moreover, since $\mathbb{Q}(E[n])$ is a Galois extension of $\mathbb{Q}$ and $\ker \rho_{E,n}= \Gal(\overline{\mathbb{Q}}/\mathbb{Q}(E[n]))$, by the first isomorphism theorem we have $G_{E}(n) \cong \Gal(\mathbb{Q}(E[n])/\mathbb{Q})$.

We would like to know what are the possibilities for $G_{E}(n)$ as a subgroup of $\GL(C_{n})$. For some values of $n$, this can be seen in Tables \ref{tableSutherland} and \ref{tableSutherland2}. For most values of $n$ we do not have a list of possibilities of $G_{E}(n)$, but we have a result that helps us see if for a given matrix subgroup $M$ of $\GL(C_{n})$ there exists an elliptic curve $E/\mathbb{Q}$ such that $\rho_{E,n}(\Gal(\overline{\mathbb{Q}}/\mathbb{Q}))=M$ (up to conjugation).
\\
\noindent The following lemma is well-known and it will be useful in cases where we analyze torsion group of elliptic curve $E/\mathbb{Q}$ over a maximal real subfield of $\mathbb{Q}(\zeta_p)$.
\begin{lemma} \label{oddtors}
Let $E/\Q$ be an elliptic curve and $L/K$ a quadratic extension of number fields with $L = K(\sqrt{d})$. Let $E(K)_{(2')}$ be the group of $K$-rational points of $E$ of odd order. Then we have:
\[
E(K(\sqrt{d}))_{(2')} \cong E(K)_{(2')} \oplus E^{d}(K)_{(2')}.
\]
\end{lemma}
Since all cyclotomic extensions are Galois over $\mathbb{Q}$, the following result imposes restrictions on the possibilities for torsion subgroup of $E/\mathbb{Q}$ over a cyclotomic fields.
\begin{lemma} \label{nika}
Let $E/\mathbb{Q}$ be an elliptic curve, $m, n \in \mathbb{N}$ and $K$ a finite Galois extension of $\mathbb{Q}$. Let $E(K)[mn] \cong C_{m}\oplus C_{mn}$ and $P \in E(K)$ point of order $mn$.
Then we have: $$[\mathbb{Q}(mP) : \mathbb{Q}] \mid M(\phi(n), [K : \mathbb{Q}]),$$ where $M(\cdot, \cdot)$ is the greatest common divisor and $\phi$ is the Euler function.
\end{lemma}
\begin{proof}
Let $P$ be a point of order $mn$ with coordinates in $K$. Then we can take $Q \in E[mn]$ such that $\{P, Q\}$ is a basis for $E[mn]$. Consider the Galois representation modulo $n$ with respect to $E$:
    $$\rho : \Gal(\overline{\mathbb{Q}}/\mathbb{Q}) \to \GL_{2}(\mathbb{Z}/n\mathbb{Z}).$$
    Let $\sigma \in \Gal(K/\mathbb{Q})$. Then we have $P^{\sigma} = \alpha P + \beta Q$ for some $\alpha, \beta \in C_{mn}$ because the action of $\sigma$ on $P$ preserves the order of a point.

\noindent Now we have $P^{\sigma} - \alpha P = \beta Q$, so $\beta Q \in E(K)$. From that follows $m\beta \equiv 0 \pmod{mn}$. Multiplying by $m$ gives us $(mP)^{\sigma} = \alpha (mP)$ so $(mP)^{\sigma} \in \langle mP \rangle$ for all $\sigma \in \Gal(K/\mathbb{Q})$. Because of preserving the order, $\alpha$ has to be in $(\mathbb{Z}/n\mathbb{Z})^{\times}$.

\noindent Since by considering the restriction map we get $\Gal(K/\mathbb{Q}) \cong \Gal(\overline{\mathbb{Q}}/\mathbb{Q})/\Gal(\overline{\mathbb{Q}}/K)$, we have $(mP)^{\sigma} \in \langle mP \rangle$ for all $\sigma \in \Gal(\overline{\mathbb{Q}}/\mathbb{Q})$. Therefore, for all $\sigma \in \Gal(\overline{\mathbb{Q}}/\mathbb{Q})$:
    $$\rho(\sigma) =
    \begin{pmatrix} 
    \varphi(\sigma) & \tau(\sigma) \\
    0 & \psi(\sigma) 
    \end{pmatrix},$$
    where $\varphi, \psi, \tau : \Gal(\overline{\mathbb{Q}}/\mathbb{Q}) \rightarrow C_{n}$, and $\varphi, \psi$ are homomorphisms with image in $(C_{n})^{\times}$.
    We know that $(mP)^{\sigma} = g(mP) \Leftrightarrow \varphi(\sigma) = g$, for all $\sigma \in \Gal(\overline{\mathbb{Q}}/\mathbb{Q})$. Therefore, we have: $$|Im(\varphi)| = |\{(mP)^{\sigma} : \sigma \in \Gal(K/\mathbb{Q})\}| = |Orb(mP)|.$$

\noindent It is clear that Stab$(mP) = \Gal(K/\mathbb{Q}(mP))$, so by orbit and stabilizer theorem we have:
    $$|Im(\varphi)| = \frac{|\Gal(K/\mathbb{Q})|}{|\Gal(K/\mathbb{Q}(mP))|} = [\mathbb{Q}(mP) : \mathbb{Q}].$$
    On the other hand, we have $Im(\varphi) \leq (\mathbb{Z}/n\mathbb{Z})^{\times}$, so we have: $$[\mathbb{Q}(mP) : \mathbb{Q}] \mid \phi(n).$$ $[\mathbb{Q}(mP) : \mathbb{Q}] \mid [K : \mathbb{Q}]$ is obvious and the proof is complete.
\end{proof}
One of the crucial results that we will need is the main result from \cite{chou}:
\begin{tm} \label{eva}
Let $E/\mathbb{Q}$ be a rational elliptic curve. Then $E(\mathbb{Q}^{ab})_{tors}$ is isomorphic to the one of the following groups:
    \begin{gather*}
    C_{m}, \quad m = 1, 3, 5, 7, 9, 11, 13, 15, 17, 19, 21, 25, 27, 37, 43, 67, 163\\
    C_{2} \oplus C_{2m}, \quad m = 1, 2, \dots, 8, 9\\
    C_{3} \oplus C_{3m}, \quad m = 1, 3\\
    C_{4} \oplus C_{4m}, \quad m = 1, 2, 3, 4\\
    C_{5} \oplus  C_{5},\\
    C_{6} \oplus C_{6},\\
     C_{8} \oplus  C_{8}.
\end{gather*}
\end{tm}

\noindent This means that all of our candidate torsion subgroups are the subgroups of the groups in the above list. Our approach will mainly consist of eliminating a certain set of possibilities from the list above in order to classify torsion groups of elliptic curves over a specific cyclotomic field.

\section{Torsion growth over \texorpdfstring{$\mathbb{Q}(\zeta_{16})$}{Q16}}

\noindent Assume that $E/\mathbb{Q}$ is an elliptic curve and that $C_{m} \oplus C_{mn} \subseteq E(\mathbb{Q}(\zeta_{16}))_{tors}$. By the properties of the Weil pairing, we have $\mathbb{Q}(\zeta_{m}) \subseteq \mathbb{Q}(\zeta_{16})$. It follows that $m \in \{1,2,4,8 \}.$ We first eliminate a certain amount of cyclic groups listed in Theorem \ref{eva}.

\begin{lemma}
Let $E/\Q$ be an elliptic curve. Then $E(\Q(\zeta_{16}))_\tors$ is not isomorphic to $C_{n}$ if \[n \in \{11, 14, 17, 18, 19, 21, 25, 27, 37, 43, 67, 163 \}.\]
\end{lemma}
\begin{proof}
Lemma \ref{nika} gives us that if $P_{n}$ is a point of order $n  \not\in \{17, 21, 25, 37 \}$, we have $[\Q(P_{n}) : \Q] \mid 2$, which is impossible by Theorem \ref{najmanquadratic}. By the same lemma we get that if $P_{n}$ is a point of order $n  \in \{21, 25, 37 \}$, then we have $[\Q(P_{n}) : \Q] \mid 4$, which is impossible by \cite[Theorem 1.4]{chou2}.
\\
It remains to consider the case $n=17$. By \cite[Theorem 5.8]{growth} we conclude that the point $P_{17}$ of order $17$ cannot be defined over some strictly smaller subfield of $\Q(\zeta_{16})$. That means that all $\sigma \in \Gal(\Q(\zeta_{16}) / \Q)$ act differently on $P_{17}$. Since $\Gal(\Q(\zeta_{16}))$ has four elements $\sigma$ such that $\sigma^2 = id$, we have that $P_{17}^{\sigma^2} = k^2P_{17} = P_{17}$ for four different $\sigma$. That means that we have $k^2 \equiv 1 \pmod{17}$ for four different $k$, a contradiction.
\end{proof}

After eliminating plenty of cyclic groups, we discuss the cases when $E$ obtains full $2$-torsion over $\mathbb{Q}(\zeta_{16}).$ This is done by the following lemmas:

\begin{lemma}
Let $E/\Q$ be an elliptic curve. Then $E(\Q(\zeta_{16}))_\tors$ is not isomorphic to $C_{2} \oplus C_{14}$ or $C_{2} \oplus C_{18}$.
\end{lemma}
\begin{proof}
\noindent We will prove the result for $C_{2} \oplus C_{14}$ and the proof for the case $C_{2} \oplus C_{18}$ is identical.
\\
\noindent Let $P_{14} \in E(\Q(\zeta_{16}))$ be the point of order $14$. From Lemma \ref{nika} we get that $[\mathbb{Q}(2P_{14}) : \mathbb{Q}] \mid 2$. It is also well-known that $[\mathbb{Q}(E[2]) : \mathbb{Q}] \in \{1,2,3,6\}$. Since $E[2]$ is defined over $\Q(\zeta_{16})$, we have $[\mathbb{Q}(E[2]) : \mathbb{Q}] \in \{1,2\}$.
\\
\noindent Let $Q_{2}$ be a point of order $2$ different from $7P_{14}$. We now have $[\mathbb{Q}(2P_{14}, 7P_{14}, Q_{2}) : \mathbb{Q}] \mid 4$. Since $2P_{14}, 7P_{14}$ and $Q_{2}$ generate our torsion subgroup $C_{2} \oplus C_{14}$, we now know that this torsion subgroup appears over some strictly smaller subfield of $\Q(\zeta_{16})$.
\\
Now we get a contradiction by using Theorem \ref{najmanquadratic} and \cite[Theorem 1.4]{chou2}.
\end{proof}

\begin{lemma}
Let $E/\Q$ be an elliptic curve. Then $E(\Q(\zeta_{16}))_\tors$ is not isomorphic to $C_{15}$,  $C_{2} \oplus C_{12}$, $C_{4} \oplus C_{12}$, $C_{4} \oplus C_{8}$, $C_{4} \oplus C_{16}$ or $C_{8} \oplus C_{8}$.
\end{lemma}
\begin{proof}
\noindent Both $X_1(15)$ and $X_1(2,12)$ are elliptic curves. A computation in Magma shows that $X_1(15)(\Q(\zeta_{16}))$ and $X_1(15)(\Q)$ have the same Mordell-Weil group structure. Therefore, since $C_{15} \not\in \Phi(1)$, it also cannot appear over $\Q(\zeta_{16})$.
\\
 The curves $X_1(2, 12)(\Q(\zeta_{16}))$ and $X_1(2, 12)(\Q(i))$ have the same Mordell-Weil group structure. It was proven in \cite[Lemma 7]{najCyclo} that $C_{2} \oplus C_{12}$ does not appear as a torsion subgroup over $\Q(i)$. Therefore, it also cannot appear over $\Q(\zeta_{16})$. This also covers the case $C_{4} \oplus C_{12}$.
 \\
 \noindent We consider the modular curves $X_1(4, 8)(\Q(\zeta_{16}))$ and $X_1(4, 8)(\Q(\zeta_{8}))$ which are actually elliptic curves. A computation in Magma shows that $X_1(4, 8)(\Q(\zeta_{16}))$ has rank $0$ and the same torsion as $X_1(4, 8)(\Q(\zeta_{8}))$, which contains only cusps, see \cite[Case 6.11]{najman_bruin}. This also covers the cases $C_{8} \oplus C_{8}$ and $C_{4} \oplus C_{16}$.
\end{proof}

The following lemma is a bit more complicated than the previous ones. The idea is to consider the corresponding modular curve $X_{1}(16)$ and its Jacobian $J_{1}(16)$ over some cyclotomic fields in order to determine that the Jacobian has rank $0$. After that, we determine torsion of $J_{1}(\mathbb{Q}(\zeta_{16}))$ and consequently the number of points on $X_{1}(16)(\mathbb{Q}(\zeta_{16}))$, all of which turn out to be cusps.

\begin{lemma}
Let $E/\Q$ be an elliptic curve. Then $E(\Q(\zeta_{16}))_\tors$ is not isomorphic to $C_{16}$, $C_{2} \oplus C_{16}$ or $C_{13}$.
\end{lemma}
\begin{proof}
\noindent We consider the modular curve $X_1(16)(\Q(\zeta_{16}))$ and its Jacobian $J_1(16)(\Q(\zeta_{16}))$. We will demonstrate the use of standard methods for determining points on $X_1(16)(\Q(\zeta_{16}))$.
\\
A computation in Magma shows that $r(J_1(16)(\Q(\zeta_{16}))) = 0$. Since
\[r(J_1(16)(\Q(\zeta_{8})(\sqrt{\zeta_{8}}))) = r(J_1(16)(\Q(\zeta_{8}))) + r(J_1^{\zeta_{8}}(16)(\Q(\zeta_{8}))),\]
 the computation of the rank becomes shorter and we obtain that the rank of our Jacobian is $0$.
\\
Now we determine $J_1(16)(\Q(\zeta_{16}))_{tors}$. Rational prime $p = 17$ splits completely in $\Q(\zeta_{16})$ so by reducing modulo some $\mathfrak p$ that lies above $p$ we get an injection
\[red_{\mathfrak p}:J_1(16)(\Q(\zeta_{16}))_{tors}\to J_1(16)(\F_{17}).\]
This map is injective due to the result of Katz \cite{katz}.
\\
A computation in Magma shows that $|J_1(16)(\F_{17})| = 400.$ It follows that $|J_1(16)(\Q(\zeta_{16}))| \le 400$. By using the generators of the $2$-torsion subgroup of $J_1(16)(\Q(\zeta_{16}))$ and some elements of $J_1(16)(\Q(\zeta_{16}))$ that we get from some known points on $X_1(16)(\Q(\zeta_{16}))$, we are able to generate a group with $400$ elements. Therefore, we know exactly how $J_1(16)(\Q(\zeta_{16}))$ looks like.
\\
Now we are able to determine all points on $X_1(16)(\Q(\zeta_{16}))$ by considering the Mumford representations of the elements of $J_1(16)(\Q(\zeta_{16}))$. We easily get that $|X_1(16)(\Q(\zeta_{16}))| = 14$ with all points being cusps.
Therefore, we can conclude that there are no elliptic curves $E/\Q(\zeta_{16})$ (and consequently $E/\Q$) with a point of order $16$ over $\Q(\zeta_{16})$.
\\It remains to show that $E(\Q(\zeta_{16}))_\tors$ can't be $C_{13}$.
We consider the modular curve $X_1(13)(\Q(\zeta_{16}))$ and its Jacobian $J_1(13)(\Q(\zeta_{16}))$.

\noindent A computation in Magma shows that $r(J_1(13)(\Q(\zeta_{16}))) = 0$. As in the previous lemma, we obtain:
\[r(J_1(13)(\Q(\zeta_{16}))) = r(J_1(16)(\Q(\zeta_{8}))) + r(J_1^{\zeta_{8}}(16)(\Q(\zeta_{8}))) = 0.\]

The next step is to determine $J_1(13)(\Q(\zeta_{16}))_{tors}$. We determine the two-torsion subgroup, which turns out to be trivial. Using the result of Katz \cite{katz}, we get an injection:
\[red_{\mathfrak p}:J_1(13)(\Q(\zeta_{16}))_{tors}\to J_1(13)(\F_{17}).\]
We also get that rational prime $q = 41$ has inertia degree $2$ in $\Q(\zeta_{16})$ so we have another injection:
\[red_{\mathfrak q}:J_1(13)(\Q(\zeta_{16}))_{tors}\to J_1(13)(\F_{41^2}).\]
We notice that $gcd(\#J_1(13)(\F_{17}), \#J_1(13)(\F_{41^2})) = 76$, so $\#J_1(13)(\Q(\zeta_{16})) \mid 76$.
\\
Since the two torsion subgroup is trivial, we get that $\#J_1(13)(\Q(\zeta_{16})) \mid 19$. We can find a point of order $19$ on our Jacobian. By checking the Mumford representations of those divisors, we find that all of the points on the Jacobian come from cusps on $X_1(13)(\Q(\zeta_{16}))$ (and actually $X_1(13)(\Q))$.
\\
\noindent Therefore, we can conclude that there are no elliptic curves $E/\Q(\zeta_{16})$ (and consequently $E/\Q$) such that $E(\Q(\zeta_{16}))_\tors \cong C_{13}$.
\end{proof}
\section{Torsion growth over \texorpdfstring{$\mathbb{Q}(\zeta_{27})$}{Q27}}
In this section we prove Theorem \ref{zeta27} using a series of lemmas. First we eliminate some possibilities for a cyclic group to appear as the subgroup of $E(\mathbb{Q}(\zeta_{27}))$. \noindent Assume that $E/\mathbb{Q}$ is an elliptic curve and that $C_{m} \oplus C_{mn} \subseteq E(\mathbb{Q}(\zeta_{27}))_{tors}$. By the properties of the Weil pairing and taking the Theorem \ref{eva} into account, we have $\mathbb{Q}(\zeta_{m}) \subseteq \mathbb{Q}(\zeta_{27})$. It follows that $m \in \{1,2,3,6 \}.$ 

\begin{lemma}
Let $E/\Q$ be an elliptic curve. Then $E(\Q(\zeta_{27}))_{tors}$ is not isomorphic to $C_{n}$ if $n \in \{11, 13, 14, 15, 16, 17, 19,25, 37,43, 67, 163\}$.
\end{lemma}
\begin{proof}
\fbox{$C_{11},  C_{25}$}: If $n \in \{11,25\}$ is a point of order $n$, then then by Lemma \ref{nika} we have $[\Q(P_{n}) : \Q] | 2$, which is impossible by Theorem \ref{najmanquadratic}. \\
\fbox{$C_{13}$}: Let $P_{13} \in E(\Q(\zeta_{27}))$ be the point of order $13$. By Lemma \ref{nika} we have
$[\Q(P_{13}) : \Q] \mid 6$. Therefore, this torsion subgroup is defined over $\Q(\zeta_{9})$.
\noindent Theorem \ref{najmanquadratic} tells us that this torsion subgroup cannot be defined over quadratic field. Therefore, it is defined over sextic or cubic field. Assume it is defined over sextic field (entire $\Q(\zeta_{9})$).
\noindent Then we can use Lemma \ref{oddtors} to get:
\[
C_{13} \cong E(\Q(\zeta_{9}))_{(2')} \cong E(\Q(\zeta_{9}^{+}))_{(2')} \oplus E^{-3}(\Q(\zeta_{9}^{+}))_{(2')}.
\]
This means that either $E$ or $E^{-3}$ has torsion subgroup $\Z / 13\Z$ defined over $\Q(\zeta_{9})^{+}$.

\noindent Now we will be finished if we prove that torsion subgroup $\Z / 13\Z$ cannot appear over $\Q(\zeta_{9})^{+}$. To do this, we consider $X_1(13)(\Q(\zeta_{9})^{+})$.
\\ As before, we use Magma to determine that $J_1(13)(\Q(\zeta_{9})^{+}) \cong J_1(13)(\Q)$ and that all points on the Jacobian come from cusps, which completes the proof.
\\
\fbox{$C_{14}$}: Let $P_{14} \in E(\mathbb{Q}(\zeta_{27}))$ be a point of order $14$. By Lemma \ref{nika} it follows that $[\mathbb{Q}(2P_{14}):\mathbb{Q}]$ divides $6$, so $\mathbb{Q}(2P_{14})$ is contained in $\mathbb{Q}(\zeta_{9})$. The point $7P_{14}$ of order $2$ satisfies $[\mathbb{Q}(7P_{14}):\mathbb{Q}] \in \{1,2,3\}$, which means that it is also contained in $\mathbb{Q}(\zeta_{9})$. It follows that $P_{14} \in E(\mathbb{Q}(\zeta_{9})).$
Consider the modular curve $X_{1}(14)$. It is an elliptic curve with LMFDB label \href{https://www.lmfdb.org/EllipticCurve/Q/14/a/5}{14.a5}. On the LMFDB page of the mentioned curve we can see that its torsion subgroup does not grow in any number field contained in $\mathbb{Q}(\zeta_{9})$. It remains to show that $r(E(\mathbb{Q}))=r(E(\mathbb{Q}(\zeta_{9}))=0$, which turns out to be true by a computation in Magma \cite{magma}. Therefore $X_{1}(14)(\mathbb{Q})=X_{1}(14)(\mathbb{Q}(\zeta_{9}))$ and there does not exist an elliptic curve over $\mathbb{Q}$ with a point of order $14$ over $\mathbb{Q}(\zeta_{9})$ and consequently over $\mathbb{Q}(\zeta_{27})$.
\\
\fbox{$C_{15}$}: Let $P_{15} \in E(\mathbb{Q}(\zeta_{27}))$ be a point of order $15$. Then $3P_{15}$ is a point of order $5$ and $[\mathbb{Q}(3P_{15}):\mathbb{Q}]$ is a divisor of $[\mathbb{Q}(\zeta_{27}):\mathbb{Q}]=18$. By Table \ref{tableSutherland} we see that $[\mathbb{Q}(3P_{15}):\mathbb{Q}] \in \{1,2\}$. The same way as in the Lemma \ref{lema43}, $C_{2} \oplus C_{12}$ case, we see that the point $5P_{15}$ of order $3$ is also defined over at most quadratic extension contained in $\mathbb{Q}(\zeta_{27})$. Since there is only one quadratic extension contained in $\mathbb{Q}(\zeta_{27})$, namely $\mathbb{Q}(\zeta_{3})$, we have $P_{15} \in E(\mathbb{Q}(\zeta_{3}))$, which contradicts the Theorem \ref{najmanquadratic}.
\\
\fbox{$C_{16}$}: Let $P_{16} \in E(\mathbb{Q}(\zeta_{27}))$ be a point of order $16$. Then the point $8P_{16}$ has order $2$ and we have $[\mathbb{Q}(8P_{16}):\mathbb{Q}] \in \{ 1,2,3\}$. By \cite[Proposition 4.8.]{growth} we have $[\mathbb{Q}(P_{16}):\mathbb{Q}]=2^{a}3^{b}$, where $a \ge 0$ is an integer and $b \in \{0,1\}.$ Since the field $\mathbb{Q}(P_{16})$ is contained in $\mathbb{Q}(\zeta_{27})$, it follows that $2^{a}3^{b}$ divides $[\mathbb{Q}(\zeta_{27}):\mathbb{Q}]=18$. We conclude that $[\mathbb{Q}(P_{16}):\mathbb{Q}] \in \{ 1,2,3,6 \}.$ Assume that $[\mathbb{Q}(8P_{16}):\mathbb{Q}]=3$. This means that $\mathbb{Q}(8P_{16})$ is cyclic, so the entire $2$-torsion is contained in this field and $P_{16}$ is defined over a number field of degree $3$ or $6$, but such a field is contained in $\mathbb{Q}(\zeta_9)$. Therefore we have $C_{2} \oplus C_{16} \subseteq E(\mathbb{Q}(\zeta_9)),$ but this is impossible by \cite[Theorem 1.1]{tomi1}. It remains to consider the case when $[\mathbb{Q}(8P_{16}):\mathbb{Q}] \in \{ 1,2 \}$. It follows that $[\mathbb{Q}(P_{16}):\mathbb{Q}] \in \{ 1,2 \}$ again by \cite[Proposition 4.8.]{growth}. Since $\mathbb{Q}(P_{16})$ is at most quadratic extension of $\mathbb{Q}$ contained in $\mathbb{Q}(\zeta_{27})$, it follows that $\mathbb{Q}(P_{16}) \subseteq \mathbb{Q}(\zeta_3)$. By \cite[Theorem 1]{najman_torshypN} this turns out to be impossible.
\\
\fbox{$C_{17}, C_{37}$}: Assume that $n \in \{ 17, 37 \}$ and that $P_{n} \in E(\mathbb{Q}(\zeta_{27}))$ is a point of order $n$. By \cite[Theorem 5.8]{growth} it follows that $[\mathbb{Q}(P_{n}):\mathbb{Q}]$ is divisible by $4$, but since $\mathbb{Q}(P_{n}) \subseteq \mathbb{Q}(\zeta_{27})$, this is impossible.
\\
\fbox{$C_{19}$}: Let us consider the case when $n=19$. If $P_{19} \in E(\mathbb{Q}(\zeta_{27}))$ is a point of order $19$, then $E$ has a rational $19$-isogeny. By \cite{loz}, we have $j(E)=-2^{15} \cdot 3^{3}$. The $19$th division polynomial $f_{E,19}$ must have a root over $\mathbb{Q}(\zeta_{27})$. Using Magma, we check that this is not the case and therefore we arrive at the contradiction.
\\
\fbox{$C_{43}, C_{67}, C_{163}$}: Assume that $n \in \{43,67,163 \}$ and $P_{n} \in E(\mathbb{Q}(\zeta_{27}))[n]$. By \cite[Theorem 2.1]{loz} it follows that $[\mathbb{Q}(P_{n}):\mathbb{Q}] \ge \frac{n-1}{2} > [\mathbb{Q}(\zeta_{27}):\mathbb{Q}]=18,$ a contradiction.
\end{proof}

\begin{lemma} \label{melisa}
Let $E/\Q$ be an elliptic curve. Then $E(\Q(\zeta_{27}))_{tors}$ is not isomorphic to $C_{18}$ or $C_{2} \oplus C_{18}$.
\end{lemma}

\begin{proof}
If $E(\Q(\zeta_{27}))_{tors} \cong C_{18}$, then Lemma \ref{nika} directly gives us that this torsion subgroup is defined over a number field of degree $6$ which can only be $\Q(\zeta_{9})$.
\\
\noindent If $E(\Q(\zeta_{27}))_{tors} \cong C_{2} \oplus C_{18}$, then Lemma \ref{nika} gives us that if $P_{18} \in E(\Q(\zeta_{27}))_{tors}$ is a point of order $18$, then $2P_{18}$ is defined over $\Q(\zeta_{9})$. We know that $[\Q(E[2]) : \Q] \in \{1,2,3,6\}$, but all unique subextensions of $\Q(\zeta_{27})$ of those degrees are contained in $\Q(\zeta_{9})$, so again our torsion subgroup is defined over $\Q(\zeta_{9})$.
\noindent Now from Lemma \ref{oddtors} we get that
$$C_{9} \cong E(\Q(\zeta_{9}))_{(2')} \cong E(\Q(\zeta_{9}^{+}))_{(2')} \oplus E^{-3}(\Q(\zeta_{9}^{+}))_{(2')}.$$
Therefore, one of $E(\Q(\zeta_{9}^{+}))$ and $E^{-3}(\Q(\zeta_{9}^{+}))$ has a point of order $9$.
Let $P_2 \in E(\Q(\zeta_{9}))_{tors}$ be a point of order $2$. If $[\Q(P_2) : \Q] \in \{1,3\}$, then $P_2$ is on both $E(\Q(\zeta_{9}^{+}))$ and $E^{-3}(\Q(\zeta_{9}^{+}))$. If $[\Q(P_2) : \Q] = 2$, then there is another point $Q_2$ of order $2$ on $E$ defined over $\Q$. In any case, both $E(\Q(\zeta_{9}^{+}))$ and $E^{-3}(\Q(\zeta_{9}^{+}))$ have a point of order $2$. Finally, one of them has a point of order $18$. However, it was proved in \cite[Lemma 3.4.7]{krijanphd} that all the points on $X_{1}(18)(\mathbb{Q}(\zeta_{9})^{+})$ are cusps, which completes the proof.
\end{proof}

\begin{lemma} \label{lema43}
Let $E/\Q$ be an elliptic curve. Assume that $C_{2} \oplus C_{2n} \cong E(\Q(\zeta_{27}))_{tors}$. Then $n \in \{ 1,2,3,4 \}$. Additionally, $C_6 \oplus C_6 \not\subseteq E(\Q(\zeta_{27}))$.
\end{lemma}

\begin{proof}
From Theorem \ref{eva} it follows that $n \le 9$. We have shown that $E(\Q(\zeta_{27}))$ cannot contain a point of order $18$ in Lemma \ref{melisa}.

\fbox{$C_{2} \oplus C_{10}$}: Let $P_{5} \in E(\Q(\zeta_{27}))$ be the point of order $5$. It follows that $E$ has a rational $5$-isogeny and $[\Q(P_5):\Q] \in \{ 1, 2\}$. If $G_{E}(2) \subseteq 2B$, then by Table \ref{tableSutherland} we see that $[\Q(E[2]):\Q] \in \{ 1, 2\}$. Thus we have found two at most quadratic fields contained in $\Q(\zeta_{27})$. Since $\Q(\zeta_{27})$ has an unique quadratic subextension $F=\Q(\sqrt{-3})$, it follows that $F=\Q(E[2])=\Q(P_5)$. We conclude that $C_{2} \oplus C_{10} \subseteq E(F)_{tors}$, which is impossible by \cite[Theorem 1]{najman_torshypN}. 
    \\
    Assume that $G_{E}(2)=2Cn$. By Theorem \cite[Theorem 1.1]{zywina} it follows that $j(E)=t^2+1728$, for some $t \in \Q$. Since $E$ has a rational $5$-isogeny, by \cite[Theorem 1.3]{zywina} we have $j(E)=\frac{25(s^2+10s+5)^3}{s^5}$, for some $s \in \Q \setminus \{0\}$. It remains to find rational points on the induced curve. By \cite[Page 61]{tomi3} we see that such rational points do not exist.
    \par
\fbox{$C_{2} \oplus C_{12}$}: Let $P_{3} \in E(\Q(\zeta_{27}))$ be the point of order $3$. The extension $\Q(P_3)$ is cyclic over $\Q$ since it is a subfield of $\Q(\zeta_{27})$. By Table \ref{tableSutherland} we see that $G_{E}(3)$ must be contained in the Borel subgroup of $\GL_{2}(\mathbb{Z}/3\mathbb{Z})$. Applying Theorem \cite[Theorem 9.3]{loz} it follows that $[\Q(P_3):\Q] \in \{ 1,2 \}$, so we have $\Q(P_3) \subseteq \Q(\zeta_3)$. Assume that $G_{E}(2) \subseteq 2B$. By \cite[Proposition 4.6]{growth} it follows that $C_{2} \oplus C_{4} \subseteq E(\Q(\zeta_{3}))$. We conclude that $C_{2} \oplus C_{12} \subseteq E(\Q(\zeta_3))$, which is impossible by \cite[Theorem 1]{najman_torshypN}.
    \\Consider the case when $G_{E}(2)=2Cn$. By \cite[Proposition 4.8]{growth} it follows that the point $P_{4}$ of order $4$ is defined over cubic or sextic subfield contained in $\mathbb{Q}(\zeta_{27})$. A computation in Magma \cite{magma} shows that if a point of order $4$ is defined over cubic or sextic number field, then $G_{E}(2)=\GL_{2}(\mathbb{Z}/2\mathbb{Z})$, a contradiction.
\par
\fbox{$C_{2} \oplus C_{16}$}: By \cite[Corollary 3.5]{pprimarnatorzija} it follows that this is impossible.,
\par
\fbox{$C_{6} \oplus C_{6}$}: Since $\Q(E[6])$ is contained in $\Q(\zeta_{27})$, we have that $|G_{E}(6)|$ divides $[\Q(\zeta_{27}):\Q]=18$. Additionally, the group $G_{E}(6)$ is cyclic. If $|G_{E}(6)| < 6$, then it follows that $E$ obtains a full $6$-torsion over a number field of degree $1,2$ or $3$, but this is impossible by Theorem \ref{najmanquadratic} and Theorem  \ref{najmancubic}. Assume that $|G_{E}(6)| \in \{ 6, 9 \}$. A search in Magma \cite{magma} shows that there exists only one cyclic group $G$ with such property, namely: 
    \[G:= \Big\langle \begin{bmatrix}
    2 & 3 \\
    3 & 1
    \end{bmatrix} \Big\rangle. \] Reducing $G$ modulo $2$ and modulo $3$ we see that $G_{E}(2)=2Cn$ and $G_{E}(3)=3Cs.1.1$. By \cite[Theorem 1.1]{zywina} we have that $j(E)=t^2+1728$, for some $t \in \Q$. Similary from Theorem \cite[Theorem 1.2]{zywina} we have that $j(E)=f(s)^3$, for some rational function $f(s)$ and $s \in \Q$. A computation in Magma \cite{magma} shows that the only affine point on the elliptic curve 
    \[t^2+1728=x^3 \]
    is $(t,x)=(0,12)$. A direct computation shows that the equation $12=f(s)$ does not have a rational solution in $\Q$. Therefore $E$ cannot have $C_6 \oplus C_6$ torsion over $\Q(\zeta_{27})$.

\end{proof}

\section{Torsion growth over \texorpdfstring{$\mathbb{Q}(\zeta_{5}),\mathbb{Q}(\zeta_{7}) \; \text{and} \; \mathbb{Q}(\zeta_{11})$}{Qp}}

We note that if $p$ is a prime number and $E/\mathbb{Q}$ an elliptic curve such that $E(\Q(\zeta_p))_{tors}$ contains a subgroup isomorphic to $C_{n} \oplus C_{mn}$, then by the properties of Weil pairing we have $\Q(\zeta_n) \subseteq \Q(\zeta_p)$ which forces $n \le 2$.

\begin{lemma}\label{petsedam}
Let $E/\mathbb{Q}$ be an elliptic curve and let $p \in \{ 5,7,11\}$ be a prime number. Apart from the groups in Mazur's theorem, the group $E(\mathbb{Q}(\zeta_p))_{tors}$ can only be isomorphic to one of the following groups:
\begin{itemize}
    \item If $p=5$, \[C_{5} \oplus C_{5} \; (\href{https://www.lmfdb.org/EllipticCurve/Q/550k2}{550k2}), \quad C_{15} \; (\href{https://www.lmfdb.org/EllipticCurve/Q/50/a/2}{50a2}) \quad \text{and} \quad C_{16} \; (\href{https://www.lmfdb.org/EllipticCurve/Q/15/a/7}{15a7}).\]
    
    \item $If p=7$, \[C_{13} \; \href{https://www.lmfdb.org/EllipticCurve/Q/147/c/2}{(147c2)}, \quad C_{14} \; \href{https://www.lmfdb.org/EllipticCurve/Q/49a1}{(49a1)},\]
\[
C_{18} \; \href{https://www.lmfdb.org/EllipticCurve/Q/14a4}{(14a4)}, \quad C_{2} \oplus C_{14} \; \href{https://www.lmfdb.org/EllipticCurve/Q/49a4}{(49a4)}, \quad C_{2} \oplus C_{18} \; \href{https://www.lmfdb.org/EllipticCurve/Q/14a5}{(14a5)}.\]

    \item If $p=11$, \[C_{11} \; \href{https://www.lmfdb.org/EllipticCurve/Q/121b2}{(121b2)}, \quad C_{25} \; \href{https://www.lmfdb.org/EllipticCurve/Q/11a3}{(11a3)}, \quad C_{2} \oplus C_{10} \; \href{https://www.lmfdb.org/EllipticCurve/Q/10230bg2}{(10230bg2)}.\]
\end{itemize}
\end{lemma}
\begin{proof}
Assume that $p=5$. By \cite[Theorem 5]{najman_bruin} we conclude that the only possibilities are the ones listed in this Lemma and $C_{17}$. It is easy to rule out $C_{17}$ by using \cite[Theorem 5.8]{growth}.
\par 
Consider the case when $p=11$. If $E/\mathbb{Q}$ is an elliptic curve and if $C_{n} \oplus C_{mn} \in E(\mathbb{Q}(\zeta_{11}))$, then by the properties of the Weil pairing we have $\mathbb{Q}(\zeta_{n}) \subseteq \mathbb{Q}(\zeta_{11})$, which forces $n \in \{1,2, 11\}$. Applying the Theorem \ref{eva} we eliminate the possibility $n=11$. By \cite[Lemma 6.0.8]{tomi3}, it remains to show that the groups $C_{15}, C_{16}$ and $C_{2}\oplus C_{12}$ do not occur. We note that the set $\Phi_{\mathbb{Q}}(2)$ is described by Theorem \ref{najmanquadratic} and the description of  the set $\Phi_{\mathbb{Q}}(5)$ can be found in \cite[Theorem 1]{38}.
\begin{itemize}
    \item $C_{15}$: Assume that $P_{15} \in E(\mathbb{Q}(\zeta_{11}))$ is a point of order $15$. Obviously we have \[ E(\mathbb{Q}(\zeta_{11}))[15] \cong C_{15}.\] By Lemma \ref{nika} it follows that $[\mathbb{Q}(P_{15}):\mathbb{Q}] \in \{1,2\}$, but this contradicts the Theorem \ref{najmanquadratic}. 
    \item $C_{16}$: If $P_{16} \in E(\mathbb{Q}(\zeta_{11}))$ is a point of order $16$, then $8P_{16}$ has order $2$ and is defined over at most quadratic extension contained in $\mathbb{Q}(\zeta_{11})$, which is $\mathbb{Q}(\sqrt{-11})$. By \cite[Proposition 4.8]{growth} it follows that $[\mathbb{Q}(P_{16}):\mathbb{Q}] \in \{1,2\}.$ A computation in Magma \cite{magma} shows that $X_{1}(16)(\mathbb{Q}(\sqrt{-11})$ contains only cusps.
    \item $C_{2} \oplus C_{12}$: As in the previous case, we show that $C_{2} \oplus C_{12} \subseteq E(\mathbb{Q}(\sqrt{-11}))$. The modular curve $X_{1}(2,12)(\mathbb{Q}(\sqrt{-11}))$ has rank $0$ and same torsion as over $\mathbb{Q}$, which means that there does not exist an elliptic curve with $C_{2} \oplus C_{12}$ torsion over $\mathbb{Q}(\sqrt{-11})$.
\end{itemize}

\par It remains to consider $p=7$.
Assume that $n \in \{11,15,16,17,21,25,19,37,43,67,163,27 \}$ and that $E(\Q(\zeta_7)) \cong C_{n}$.
\begin{itemize}
    \item $n \in \{11, 15, 17, 25 \}$: Lemma \ref{nika} gives us that if $P_{n}$ is a point of order $n$, we have $[\Q(P_{n}) : \Q] \mid 2$. Now Theorem \ref{najmanquadratic} gives us the contradiction.
    \item $n \in \{19, 37, 43, 67, 163 \}$:
From \cite[Theorem 5.8]{growth}, we get that the point of order $n$ cannot be defined over the field $\Q(\zeta_{7})$ (a degree $6$ extension).
    \item $n=27$:
This follows from \cite[Theorem 1.1]{tomi1}.
    \item $n=16$:
Lemma \ref{nika} gives us that if $P_{16} \in E(\Q(\zeta_7))$ is a point of order $16$, we have $[\Q(P_{16}) : \Q] \mid 2$. That means that $P_{16} \in E(\Q(\sqrt{-7}))$. We can use the similar methods as before in Magma to consider $X_1(16)(\Q(\sqrt{-7}))$ and prove that $E$ cannot have a point of order $16$ defined over $\Q(\sqrt{-7})$.
    \item $n=21$:
For $C_{21}$, we conclude from \cite[Lemma 2.7]{chou2} that $E$ has a $\Q$-rational $21$-isogeny. There are $4$ elliptic curves (up to $\overline{\Q}$-isomorphism) with a rational $21$-isogeny (see \cite[p.78-80]{rj21}).
Therefore, we can use the division polynomial method since the elliptic curves with the same $j$-invariant have identical division polynomials, up to scalar. We will consider the seventh division polynomials. We can use Magma \cite{magma} to factor those polynomials in the field $\Q(\zeta_7)$ and see that they have no zeroes there. Hence, this case is impossible.
    
\end{itemize}
It remains to eliminate only three non-cyclic groups.
For $C_{2} \oplus C_{16}$, we can use Lemma \ref{nika} to show that if $P_{16} \in E(\Q(\zeta_7))$ is of order $16$, then $[\Q(2P_{16}) : \Q] \mid 2$. By \cite[Proposition 4.6]{growth}, we can conclude that $[\Q(P_{16}) : \Q] \in \{1, 2\}$. Hence, we have a point of order $16$ defined over $\Q(\sqrt{-7})$. However, we already proved that this can't happen when we considered $C_{16}$.
For $C_{2} \oplus C_{10}$ and $C_{2} \oplus C_{12}$, we consider modular curves $X_1(2, 10)(\Q(\zeta_{7}))$ and $X_1(2, 12)(\Q(\zeta_{7}))$ and use Magma to show that they don't have non-cuspidal points, which completes the proof.
\end{proof}

\begin{remark}
Ideally, one would like to give a useful description of possible isomorphism classes of $E/\mathbb{Q}(\zeta_p)$, where $E/\mathbb{Q}$ is an elliptic curve and $p$ is a prime number. One can start with the following question that seems to be out of reach for the authors at the time of writing this paper.
\\
Let $n \in \{ 13, 16, 18 ,25 \}$. Under what conditions on the prime number $p$ does there exist an elliptic curve $E/\mathbb{Q}$ with a point $P_n \in E(\mathbb{Q}(\zeta_p)) $ of order $n$?
\end{remark}

\newpage

\subsection{Appendix: Images of Mod \texorpdfstring{$p$}{Appendix} Galois representations associated to elliptic curves over \texorpdfstring{$\mathbb{Q}$}{Appendix}}

For each possible known subgroup $G_E(p) \subsetneq \GL_2(\mathbb{F}_p)$ where $E/\mathbb{Q}$ is a non-CM elliptic curve and $p$ is a prime, Tables \ref{tableSutherland} and  \ref{tableSutherland2} list in the first and second column the corresponding labels in Sutherland and Zywina notations, and the following data:
\begin{itemize}
\item $d_v=[G_{E}(p):G_{E}(p)_v]=|G_{E}(p).v|$ for $v\in\mathbb{F}_p^2$, $v\ne (0,0)$; equivalently, the degrees of the extensions $\mathbb{Q}(P)$ over $\mathbb{Q}$ for points $P\in E(\overline{\mathbb{Q}})$ of order $p$.
\item$d=|G_{E}(p)|$; equivalently, the degree $\mathbb{Q}(E[p])$ over $\mathbb{Q}$.
\end{itemize}

Note that Tables \ref{tableSutherland} and  \ref{tableSutherland2} are partially extracted from Table 3 of \cite{sutherlandgalrep}. The difference is that \cite[Table 3]{sutherlandgalrep} only lists the minimum of $d_v$, which is denoted by $d_1$ therein.

\begin{table}[h!] \begin{footnotesize} \begin{center}
\begin{tabular}{cc}
\begin{tabular}[t]{cllcc}
& Sutherland  &  Zywina  & $d_v$ & $d$ \\
\toprule
&\texttt{2Cs} & $G_1$   & 1 & 1\\
&\texttt{2B} &$G_2$  & 1\,,\,2 & 2\\
&\texttt{2Cn} & $G_3$  & 3  & 3\\
\midrule
&\texttt{3Cs.1.1} &$H_{1,1}$ & 1\,,\,2 & 2\\
&\texttt{3Cs} & $G_1$  & 2\,,\,4  & 4 \\
 &\texttt{3B.1.1} & $H_{3,1}$ & 1\,,\,6  &6 \\
&\texttt{3B.1.2} &  $H_{3,2}$ & 2\,,\,3  & 6\\
&\texttt{3Ns} &$G_2$  & 4  & 8\\
&\texttt{3B} & $G_3$  & 2\,,\,6  & 12\\
&\texttt{3Nn} & $G_4$  & 8  & 16\\
\midrule
&\texttt{5Cs.1.1} & $H_{1,1}$  & 1\,,\,4  & 4\\
&\texttt{5Cs.1.3} & $H_{1,2}$ & 2\,,\,4  & 4
\end{tabular} & \begin{tabular}[t]{cllcc}
& Sutherland  &  Zywina & $d_v$ & $d$ \\
\toprule
&\texttt{5Cs.4.1} & $G_1$  & 2\,,\,4\,,\,8  & 8\\
&\texttt{5Ns.2.1} & $G_3$  & 8\,,\,16  & 16\\
&\texttt{5Cs} &  $G_2$  & 4\,,\,4  & 16\\
&\texttt{5B.1.1} & $H_{6,1}$  & 1\,,\,20  & 20\\
&\texttt{5B.1.2} & $H_{5,1}$  & 4\,,\,5  & 20\\
&\texttt{5B.1.4} & $H_{6,2}$  & 2\,,\,20  & 20\\
&\texttt{5B.1.3} & $H_{5,2}$  & 4\,,\,10 & 20\\
&\texttt{5Ns} & $G_{4}$  & 8\,,\,16  & 32\\
&\texttt{5B.4.1} & $G_{6}$  & 2\,,\,20  & 40\\
&\texttt{5B.4.2} & $G_{5}$  & 4\,,\,10  & 40\\
&\texttt{5Nn} & $G_{7}$  & 24  & 48\\
&\texttt{5B} & $G_{8}$  & 4\,,\,20 & 80\\
&\texttt{5S4} & $G_{9}$ & 24  & 96
\end{tabular}\\
\midrule
\end{tabular}
\end{center} \end{footnotesize} \end{table}

\begin{table}[h!] \begin{footnotesize} \begin{center}
\begin{tabular}{cc}
\begin{tabular}[t]{cllcc}
& Sutherland  &  Zywina & $d_v$ & $d$ \\
\toprule
&\texttt{7Ns.2.1} &  $H_{1,1}$ &   $ 6\,,\,9\,,\,18$  & 18 \\
&\texttt{7Ns.3.1} & $G_{1}$ & $12\,,\,18$ & 36 \\
&\texttt{7B.1.1} & $H_{3,1}$ & $ 1\,,\,42$ & 42 \\
&\texttt{7B.1.3} & $H_{4,1}$ & $ 6\,,\, 7$  & 42 \\
 &\texttt{7B.1.2} & $H_{5,2}$ & $ 3\,,\,42$ & 42 \\
&\texttt{7B.1.5} & $H_{5,1}$ &  $ 6\,,\,21$   & 42 \\
 &\texttt{7B.1.6} & $H_{3,2}$ &  $ 2\,,\,21$ & 42 \\
 &\texttt{7B.1.4} &$H_{4,2}$ &  $ 3\,,\,14$   & 42 \\
 &\texttt{7Ns} & $G_{2}$ &  $ 12\,,\,36$ & 72 \\
  &\texttt{7B.6.1} & $G_{3}$  & $ 2\,,\,42$   & 84\\
   &\texttt{7B.6.3} & $G_{4}$  & $ 6\,,\, 14$   & 84 \\
  &\texttt{7B.6.2} & $G_{5}$  & $ 6\,,\, 42$   & 84
\end{tabular} & \begin{tabular}[t]{cllcc}
& Sutherland  &  Zywina & $d_v$ & $d$ \\
\toprule
  &\texttt{7Nn} & $G_{6}$  & $ 48$&    96 \\
   &\texttt{7B.2.1} & $H_{7,2}$&  $ 3\,,\,42$   & 126 \\
   &\texttt{7B.2.3} & $H_{7,1}$ &  $ 6\,,\, 21$   & 126 \\
   &\texttt{7B} & $G_{7}$ & $ 6\,,\, 42$ & 252 \\
\midrule
&\texttt{11B.1.4} & $H_{1,1}$  & 5\,,\,110 & 110\\
&\texttt{11B.1.5} & $H_{2,1}$  & 5\,,\,110  & 110\\
&\texttt{11B.1.6} & $H_{2,2}$  & 10\,,\,55  & 110\\
&\texttt{11B.1.7} & $H_{1,2}$  & 10\,,\,55  & 110\\
&\texttt{11B.10.4} & $G_{1}$  & 10\,,\,110  & 220\\
&\texttt{11B.10.5} & $G_{2}$  & 10\,,\,110  & 220\\
&\texttt{11Nn} & $G_{3}$ & 120 & 240
\end{tabular}\\
\bottomrule
\end{tabular}
\end{center} \end{footnotesize}
\caption{Possible images $G_E(p)\ne \GL_2(\mathbb{F}_p)$, for $p\le 11$, for non-CM elliptic curves $E/\mathbb{Q}$.}\label{tableSutherland}
\vspace{10pt}
\end{table}

\begin{table}[h!] \begin{footnotesize} \begin{center}
\begin{tabular}{cc}
\begin{tabular}[t]{cllcc}
& Sutherland & Zywina & $d_v$ & $d$ \\
\toprule
&\texttt{13S4} &$G_{7}$ & $72\,,\,96$ & 288 \\
 &\texttt{13B.3.1} &$H_{5,1}$ & $3 \,,\, 156$ & 468 \\
 &\texttt{13B.3.2} & $H_{4,1}$ & $ 12\,,\, 39$ & 468 \\
 &\texttt{13B.3.4} & $H_{5,2}$ & $ 6\,,\,156$ & 468 \\
 &\texttt{13B.3.7} & $H_{4,2}$ & $ 12 \,,\, 78$ & 468 \\
  &\texttt{13B.5.1} & $G_{2}$ & $ 4 \,,\, 156$ & 624 \\
  &\texttt{13B.5.2} & $G_{1}$ & $ 12 \,,\, 52$ & 624\\
  &\texttt{13B.5.4} & $G_{3}$ & $ 12 \,,\, 156$ & 624
\end{tabular} & \begin{tabular}[t]{cllcc}
& Sutherland & Zywina & $d_v$ & $d$ \\
\toprule
&\texttt{13B.4.1} & $G_{5}$ & $ 6 \,,\, 156$  & 936 \\
  &\texttt{13B.4.2} & $G_{4}$ & $ 12 \,,\, 78$   & 936 \\
  &\texttt{13B} & $G_{6}$ & $ 12 \,,\, 156$  & 1872 \\
\midrule
&\texttt{17B.4.2} & $G_{1}$ & $ 8 \,,\, 272$  & 1088 \\
&\texttt{17B.4.6} & $G_{2}$ & $ 16 \,,\, 136$  & 1088 \\
\midrule
 &\texttt{37B.8.1} & $G_{1}$ & 12 \,,\,  1332 & 15984 \\
 &\texttt{37B.8.2} & $G_{2}$ &  36 \,,\, 444 & 15984
\end{tabular} \\
\bottomrule
\end{tabular}
\end{center} \end{footnotesize}
\caption{Known images $G_E(p)\ne \GL_2(\mathbb{F}_p)$, for $p=13, 17$ or $37$, for non-CM elliptic curves $E/\mathbb{Q}$.}\label{tableSutherland2}
\vspace{10pt}
\end{table}

\begin{acknowledgments}
The authors gratefully acknowledge support
from the QuantiXLie Center of Excellence, a project co-financed by the Croatian Government and European Union through the
European Regional Development Fund - the Competitiveness and Cohesion Operational Programme (Grant KK.01.1.1.01.0004) and
by the Croatian Science Foundation under the project no. IP-2018-01-1313.
\end{acknowledgments}

\nocite{*}
\bibliographystyle{babplain-fl}
\bibliography{ciklo_bib}

\end{document}